\documentclass[a4paper,12pt, twoside, pdftex]{amsart}

\usepackage{fullpage}
\usepackage{amsmath, amsthm, amssymb}
\usepackage{txfonts}
\usepackage{amsfonts}
\usepackage{prettyref}
\usepackage{mathrsfs}
\usepackage[all]{xy}
\usepackage[utf8]{inputenc}

\usepackage{tikz}
\usepackage{dynkin-diagrams}
\usetikzlibrary{backgrounds}

\theoremstyle{plain}
\newtheorem{thm}{Theorem}[section]

\newtheorem{lem}[thm]{Lemma}

\newtheorem*{acknowledgements}{Acknowledgements}
\newtheorem*{notations}{Notations and conventions}

\theoremstyle{definition}

\theoremstyle{remark}
\newtheorem{rmk}{Remark}[section]

\numberwithin{equation}{section}

\newcommand{\pref}{\prettyref}

\newrefformat{th}{Theorem~\ref{#1}}
\newrefformat{cr}{Corollary~\ref{#1}}
\newrefformat{lm}{Lemma~\ref{#1}}
\newrefformat{dl}{Definition-Lemma~\ref{#1}}
\newrefformat{df}{Definition~\ref{#1}}
\newrefformat{cl}{Claim~\ref{#1}}
\newrefformat{sl}{Sublemma~\ref{#1}}
\newrefformat{pr}{Proposition~\ref{#1}}
\newrefformat{cj}{Conjecture~\ref{#1}}
\newrefformat{st}{Step~\ref{#1}}
\newrefformat{sc}{Section~\ref{#1}}
\newrefformat{df}{Definition~\ref{#1}}
\newrefformat{rm}{Remark~\ref{#1}}
\newrefformat{q}{Question~\ref{#1}}
\newrefformat{pb}{Problem~\ref{#1}}
\newrefformat{cd}{Condition~\ref{#1}}
\newrefformat{eg}{Example~\ref{#1}}
\newrefformat{he}{Heore~\ref{#1}}
\newrefformat{fg}{Figure~\ref{#1}}
\newrefformat{tb}{Table~\ref{#1}}
\newrefformat{as}{Assumption~\ref{#1}}

\newcommand{\bhom}{\mathop{\mathbf{hom}}\nolimits}
\newcommand{\bh}{\mathop{\mathbf{h}}\nolimits}
\newcommand{\lb}{\left(}
\newcommand{\rb}{\right)}

\newcommand{\Gr}{\operatorname{Gr}}
\newcommand{\LGr}{\operatorname{LGr}}

\newcommand{\cA}{\mathcal{A}}

\newcommand{\cE}{\mathcal{E}}
\newcommand{\cF}{\mathcal{F}}

\newcommand{\cO}{\mathcal{O}}

\newcommand{\bP}{\mathbb{P}}

\newcommand{\bfk}{\mathbf{k}}

\newcommand{\bfF}{\mathbf{F}}
\newcommand{\bfP}{\mathbf{P}}
\newcommand{\bfQ}{\mathbf{Q}}

\newcommand{\bfV}{\mathbf{V}}

\newcommand{\scrL}{\mathscr{L}}

\newcommand{\scrQ}{\mathscr{Q}}

\newcommand{\scrS}{\mathscr{S}}

\title{Derived equivalences for the flops of type $C_2$ and $A^G_4$ via mutation of semiorthogonal decomposition}

\author[H.~Morimura]{Hayato Morimura}
\address{
Graduate School of Mathematical Sciences,
The University of Tokyo,
3-8-1 Komaba,
Meguro-ku,
Tokyo,
153-8914,
Japan.}
\email{morimura@ms.u-tokyo.ac.jp, hmorimur@sissa.it}
\curraddr{
SISSA,
via Bonomea 265,
34136
Trieste,
Italy.}

\date{}
\begin{document}

\begin{abstract}
We give a new proof of the derived equivalence of a pair of varieties
connected by the flop of type $C_2$  in the list of Kanemitsu
\cite{1812.05392},
which is originally due to Segal \cite{Seg16}.
We also prove the derived equivalence of a pair of varieties
connected by the flop of type $A^G_4$ in the same list.
The latter proof follows that of the derived equivalence of Calabi--Yau 3-folds in Grassmannians
$\text{Gr}(2,5)$
and
$\text{Gr}(3,5)$
by Kapustka and Rampazzo \cite{1711.10231} closely.
\end{abstract}

\maketitle

\section{Introduction}
Let $G$ be a semisimple Lie group and $B$ a Borel subgroup of $G$.
For distinct maximal parabolic subgroups $P$ and $Q$ of $G$ containing $B$,
three homogeneous spaces $G / P$, $G / Q$, and $G / (P \cap Q)$
form the following diagram:
\begin{align*} 
\begin{gathered}
\xymatrix{
&
\bfF \coloneqq G / (P \cap Q)
\ar _{ \varpi _- }[dl]
\ar ^{ \varpi _+ }[dr]
&\\
\bfP \coloneqq G / P
&
&
\bfQ \coloneqq G / Q}
\end{gathered}
\end{align*}
We write the hyperplane classes of $\bfP$ and $\bfQ$
as $h$ and $H$ respectively.
By abuse of notation,
the pull-back to $\bfF$ of the hyperplane classes $h$ and $H$
will be denoted by the same symbol.
The morphisms $\varpi_-$ and $\varpi_+$ are
projective morphisms
whose relative $\cO(1)$ are $\cO(H)$ and $\cO(h)$ respectively.
We consider the diagram
\begin{align} \label{eq:zero-sections_diagram}
\begin{gathered}
\xymatrix{
& \bfF \ar[dl]_{\varpi_-} \ar@{^(->}[d]^\iota \ar[dr]^{\varpi_+} & \\
\bfP \ar@{^(->}[d]_{\iota_-} & \bfV \ar[dl]_{\varphi_-} \ar[dr]^{\varphi_+} & \bfQ \ar@{^(->}[d]^{\iota_+} \\
\bfV-  \ar[dr]^{\phi_-} & & \ar[dl]_{\phi_+} \bfV_+ \\
& \bfV_0
}
\end{gathered}
\end{align}
where
\begin{itemize}
\item
$\bfV_-$ is the total space of $\left( (\varpi_{-})_* \cO(h+H) \right)^\vee$ over $\bfP$,
\item
$\bfV_+$ is the total space of$\left( (\varpi_{+})_* \cO(h+H) \right)^\vee$ over $\bfQ$,
\item
$\bfV$ is the total space of $\cO(-h-H)$ over $\bfF$,
\item
$\iota_-, \iota_+$, and $\iota$ are the zero-sections, 
\item
$\varphi_-$ and $\varphi_+$ are blow-ups of the zero-sections, and
\item
$\phi_-$ and $\phi_+$ are the affinizations
which contract the zero sections.
\end{itemize}
If $\bfV_-$ and $\bfV_+$ have the trivial canonical bundles,
then one expects
from \cite[Conjecture 4.4]{MR1957019}
or \cite[Conjecture 1.2]{Kaw}
that $\bfV_-$ and $\bfV_+$
are derived-equivalent.

When $G$ is the simple Lie group of type $G_2$,
Ueda \cite{Ued} used sequence of mutations of semiorthogonal decompositions of $D^b(\bfV)$
obtained by applying Orlov's theorem \cite{Orl} to the diagram \prettyref{eq:zero-sections_diagram}
to prove the derived equivalence of $\bfV_-$ and $\bfV_+$.
This sequence of mutations in turn follows that
of Kuznetsov \cite{1611.08386} closely.

In this paper,
by using the same method,
we give a new proof to the following theorem,
which is originally due to Segal \cite{Seg16},
where the flop was attributed to Abuaf:

\begin{thm} \label{thm:Deq}
Varieties connected  by the flop of type $C_2$ are derived-equivalent. 
\end{thm}

The term \emph{the flop of type $C_2$} was introduced in \cite{1812.05392},
where simple K-equivalent maps in dimension at most 8 were classified. 
There are several ways to prove \prettyref{thm:Deq}.
In \cite{Seg16},
Segal showed the derived equivalence by using tilting vector bundles. 
Hara \cite{arXiv:1706.04417} constructed alternative tilting vector bundles and
studied the relation between functors defined by him and Segal. 

The flop of type $A^G_{2r-2}$ is also in the list of Kanemitsu\cite{1812.05392}.
It connects $\bfV_-$ and $\bfV_+$ for $\bfP = \Gr(r-1, 2r-1)$ and $\bfQ = \Gr(r, 2r-1)$.
Similarly, we prove the following theorem:

\begin{thm} \label{thm:Deq2}
Varieties connected by the flop of type $A^G_4$ are derived-equivalent. 
\end{thm}

Although the proof of \pref{thm:Deq2} is parallel to that of the derived equivalence of Calabi--Yau complete intersections in
$\bfP = \Gr(2, 5)$
and
$\bfQ = \Gr(3, 5)$
defined by global sections of the equivariant vector bundles dual to
$\bfV_-$
and
$\bfV_+$
in \cite[Theorem 5.7]{1711.10231},
we write down a full detail for clarity.
As explained in \cite{Ued},
the derived equivalence obtained in \cite{1711.10231} in turn follows from
\pref{thm:Deq2} using matrix factorizations.

We also give a similar proof of derived equivalences for a Mukai flop and a standard flop.
For a Mukai flop,
Kawamata \cite{Kaw} and Namikawa \cite{Nam} independently showed
the derived equivalence
by using the pull-back and the push-forward
along the fiber product $\bfV_- \times_{\bfV_0} \bfV_+$.
Addington, Donovan, and Meachan \cite{arXiv:1507.02595}
introduced a generalization
of the functor of Kawamata and Namikawa
parametrized by an integer,
and discovered that certain compositions of these functors
give the $\bP$-twist
in the sense of Huybrechts and Thomas \cite{HT}.
They also considered the case of a standard flop,
where the derived equivalence is originally proved
by Bondal and Orlov \cite{alg-geom/9506012}.
Our proof is obtained by proceeding the mutation performed in
\cite{alg-geom/9506012}
and
\cite{arXiv:1507.02595}
a little further in a straightforward way.
Hara \cite{Hara} also studied a Mukai flop
in terms of non-commutative crepant resolutions.

For a standard flop,
Segal \cite{Seg11} showed the derived equivalence
by using
the grade restriction rule
for variation of geometric invariant theory quotients (VGIT)
originally introduced by Hori, Herbst, and Page \cite{0803.2045}.
VGIT method was subsequently developed
by Halpern-Leistner \cite{HL}
and Ballard, Favero, and Katzarkov \cite{arXiv:1203.6643}.
It is an interesting problem to develop this method further
to prove the derived equivalence
for the flop of type $C_2$ and $A^G_4$,
and a Mukai flop.

\begin{notations}
We work over an algebraically closed field $\mathbf{k}$ of characteristic 0 throughout this paper.
All pull-back and push-forward are derived unless otherwise specified.
The complexes underlying $\operatorname{Ext}^\bullet(-,-)$ and $\operatorname{H}^\bullet(-)$ will be denoted by $\bhom(-,-)$ and $\mathbf{h}(-)$ respectively.
\end{notations}

\begin{acknowledgements}
The author would like to express his gratitude to Kazushi Ueda for guidance and encouragement.
The author would like to thank anonymous reviewers for their careful reading of the manuscript and their many suggestions and comments.
The author declare that he has no conflict of interest.
\end{acknowledgements}

\section{flop of type $C_2$}

Let $P$ and $Q$ be the parabolic subgroups of
the simple Lie group $G$ of type $C_2$
associated with the crossed Dynkin diagrams
\dynkin{C}{x*}
and
\dynkin{C}{*x}.
The corresponding homogeneous spaces are
the projective space
$
 \bfP = \bP(V),
$
the Lagrangian Grassmannian
$
 \bfQ = \LGr(V),
$
and the isotropic flag variety
$
 \bfF
  = \bP_\bfP \left( \scrL_\bfP^\perp / \scrL_\bfP \right)
  = \bP_\bfQ \left( \scrS_\bfQ \right).
$
Here $V$ is a $4$-dimensional symplectic vector space,
$\scrL_\bfP^\perp$ is the rank 3 vector bundle
given as the symplectic orthogonal
to the tautological line bundle
$\scrL_\bfP \cong \cO_\bfP(-h)$ on $\bfP$,
and $\scrS_{\bfQ}$ is the tautological rank 2 bundle on $\bfQ$.
Note that $\bfQ$ is also a quadric hypersurface in $\bP^4$.
Tautological sequences on $\bfQ = \LGr(V)$ and $\bfF \cong \bP_\bfQ \lb \scrS_\bfQ \rb$ give
\begin{align} \label{eq:cSES1}
 0 \to \scrS_\bfQ \to \cO_\bfQ \otimes V \to \scrS^\vee_\bfQ \to 0
\end{align}
and
\begin{align} \label{eq:cSES}
 0 \to \cO_\bfF(-h+H) \to \scrS_\bfF^\vee \to \cO_\bfF(h) \to 0,
\end{align}
where
$
 \scrS_\bfF \coloneqq \varpi_+^* \scrS_\bfQ.
$
We have
\begin{align*}
 (\varpi_-)_* \left( \cO_\bfF(H) \right)
 \cong \left( \left( \scrL_\bfP^\perp / \scrL_\bfP \right) \otimes \scrL_\bfP \right)^\vee
\end{align*}
and
\begin{align*}
 (\varpi_+)_* \left( \cO_\bfF(h) \right)
 \cong \scrS_\bfQ^\vee,
\end{align*}
whose determinants are given by
$\cO_\bfP(2h)$ and
$\cO_\bfQ(H)$ respectively.
Since
$\omega_{\bfP} \cong \cO_\bfP(-4h)$,
$\omega_{\bfQ} \cong \cO_\bfQ(-3H)$,
and
$\omega_{\bfF} \cong \cO_\bfF(-2h-2H)$,
we have
$\omega_{\bfV_-} \cong \cO_{\bfV_-}$,
$\omega_{\bfV_+} \cong \cO_{\bfV_+}$,
and
$\omega_{\bfV} \cong \cO_\bfV(-h-H)$.

Recall from \cite{Bei} that 
\begin{align} \label{eq:cEC1}
D^b(\bfP) = \langle \cO_{\bfP}(-2h), \cO_\bfP(-h), \cO_\bfP, \cO_\bfP(h) \rangle,
\end{align}
and from \cite{Kuz} (cf. also \cite{Kap}) that
\begin{align*} 
D^b(\bfQ) = \langle \cO_{\bfQ}(-H), \mathscr{S}_{\bfQ}^{\vee}(-H), \cO_\bfQ, \cO_\bfQ(H) \rangle.
\end{align*}
Since $\varphi_\pm$ are blow-ups along the zero-sections,
it follows from \cite{Orl} that 
\begin{align} \label{eq:cSOD1}
D^b(\bfV) = \langle \iota_* \varpi_-^* D^b(\bfP), \Phi_-(D^b(\bfV_-)) \rangle
\end{align}
and
\begin{align} \label{eq:cSOD2}
D^b(\bfV) = \langle \iota_* \varpi_+^* D^b(\bfQ), \Phi_+(D^b(\bfV_+)) \rangle,
\end{align}
where
\begin{align*}
\Phi_- \coloneqq \left( (-) \otimes \cO_\bfV(H) \right) \circ \varphi_-^* \colon D^b(\bfV_-) \to D^b(\bfV)
\end{align*}
and
\begin{align*}
\Phi_+ \coloneqq \left( (-) \otimes \cO_\bfV(h) \right) \circ \varphi_+^* \colon D^b(\bfV_+) \to D^b(\bfV).
\end{align*}
By abuse of notation,
we use the same symbol for an object of $D^b(\bfF)$
and its image in $D^b(\bfV)$ by the push-forward $\iota_*$.
\prettyref{eq:cEC1} and \prettyref{eq:cSOD1} give
\begin{align*}
D^b(\bfV) = \langle \cO_\bfF(-2h), \cO_\bfF(-h), \cO_\bfF, \cO_\bfF(h), \Phi_-(D^b(\bfV_-)) \rangle.
\end{align*}
Since $\omega_\bfV \cong \cO_\bfV(-h-H)$,
by mutating the first term to the far right,
and then
$\Phi_-(D^b(\bfV_-))$ one step to the right,
we obtain
\begin{align*}
D^b(\bfV) = \langle \cO_\bfF(-h), \cO_\bfF, \cO_\bfF(h), \cO_\bfF(-h+H), \Phi_1(D^b(\bfV_-)) \rangle,
\end{align*}
where
\begin{align*}
\Phi_1 \coloneqq R_{\langle \cO_\bfF(-h+H) \rangle} \circ \Phi_-.
\end{align*}
In the sequel,
we will use the following fact.

\begin{lem} \label{lem:homs}
Given two vector bundles $\cE_\bfF, \cF_\bfF$ on $\bfF$,
if 
$\bh \left( \cE^\vee_\bfF \otimes \cF_\bfF (-h-H) \right) \simeq 0$,
then we have
$\bhom_{\cO_\bfV} \left( \cE_{\bfF}, \cF_{\bfF} \right)
\simeq
\bh \left( \cE^\vee_{\bfF} \otimes \cF_{\bfF} \right)$.
\end{lem}
\begin{proof}
We have
\begin{align*}
\bhom_{\cO_\bfV} \left( \cE_{\bfF}, \cF_{\bfF} \right)
&\simeq
\bhom_{\cO_\bfV} \left( \left\{ \cE_{\bfV}(h+H) \to \cE_{\bfV} \right\} , \cF_{\bfF} \right) \\
&\simeq
\bh \left( \left\{ \cE^\vee_{\bfF} \otimes \cF_{\bfF} \to \cE^\vee_{\bfF} \otimes \cF_{\bfF}(-h-H) \right\} \right) \\
&\simeq
\bh \left( \cE^\vee_{\bfF} \otimes \cF_{\bfF} \right).
\end{align*}
\end{proof}

Note that
the canonical extension of $\cO_\bfF(h)$ by $\cO_\bfF(-h+H)$
associated with
\begin{align*}
\bhom_{\cO_\bfV} \left( \cO_{\bfF}(h), \cO_{\bfF}(-h+H) \right)
&\simeq
\bh \left( \cO_\bfF(-2h+H) \right) \\
&\simeq
\bh \left( (\varpi_+)_* \cO_\bfF(-2h) \otimes \cO_\bfQ(H) \right) \\
&\simeq
\bh \left( \cO_\bfQ [-1] \right) \\
&\simeq
\bfk[-1]
\end{align*}
is given by the short exact sequence \prettyref{eq:cSES}.
By mutating $\cO_\bfF(-h+H)$ one step to the left,
$\cO_{\bfF}(-h)$ to the far right,
and then
$\Phi_1(D^b(\bfV_-))$ one step to the right,
we obtain
\begin{align*}
D^b(\bfV) = \langle \cO_\bfF, \scrS_\bfF^\vee, \cO_\bfF(h), \cO_\bfF(H), \Phi_2(D^b(\bfV_-)) \rangle,
\end{align*}
where
\begin{align*}
\Phi_2 \coloneqq R_{\langle \cO_\bfF(H) \rangle} \circ \Phi_1.
\end{align*}
One can easily see that $\cO_\bfF(h)$ and $\cO_\bfF(H)$ are orthogonal,
so that
\begin{align}
D^b(\bfV) = \langle \cO_\bfF, \scrS_\bfF^\vee, \cO_\bfF(H), \cO_\bfF(h), \Phi_2(D^b(\bfV_-)) \rangle.
\end{align}
By mutating $\Phi_2(D^b(\bfV_-))$ one step to the left,
and then
$\cO_{\bfF}(h)$ to the far left,
we obtain
\begin{align*}
D^b(\bfV) = \langle \cO_\bfF(-H), \cO_\bfF, \scrS_\bfF^\vee, \cO_\bfF(H), \Phi_3(D^b(\bfV_-)) \rangle,
\end{align*}
where
\begin{align*}
\Phi_3 \coloneqq L_{\langle \cO_\bfF(h) \rangle} \circ \Phi_2.
\end{align*}
We have
\begin{align*}
 \bhom_{\cO_\bfV} \lb \cO_\bfF, \scrS_\bfF^\vee \rb
\simeq
\bh \lb \scrS_\bfF^\vee \rb 
\simeq V^\vee,
\end{align*}
and
the dual of \prettyref{eq:cSES1} shows that
the kernel of the evaluation map
$
 \cO_\bfF \otimes V^\vee \to \scrS_\bfF^\vee
$
is $\scrS_\bfF \cong \scrS^\vee_\bfF (-H)$.
By mutating $\scrS_\bfF^\vee$ one step to the left,
we obtain
\begin{align} \label{eq:cSOD4}
D^b(\bfV) = \langle \cO_\bfF(-H), \scrS_{\bfF}^\vee (-H), \cO_\bfF, \cO_\bfF(H), \Phi_3(D^b(\bfV_-)) \rangle.
\end{align}
By comparing \eqref{eq:cSOD4} with \eqref{eq:cSOD2},
we obtain a derived equivalence
\begin{align*} 
\Phi \coloneqq \Phi_+^! \circ \Phi_3 \colon D^b(\bfV_-) \xrightarrow{\sim} D^b(\bfV_+),
\end{align*}
where
\begin{align*}
\Phi_+^!(-) \coloneqq (\varphi_+)_* \circ \left( (-) \otimes \cO_\bfV(-h) \right)
\colon D^b(\bfV) \to D^b(\bfV_+)
\end{align*}
is the left adjoint functor of $\Phi_+$.

\section{flop of type $A^G_4$}
Let $P$ and $Q$ be the parabolic subgroups of the simple Lie group $G$ of type $A_4$ associated with the crossed Dynkin diagrams
\dynkin{A}{*x**} and \dynkin{A}{**x*}.
The corresponding homogeneous spaces are the Grassmannians
$\bfP = \Gr(2, V)$,
$\bfQ = \Gr(3, V)$,
and
the partial flag variety
$\bfF = \bP_\bfP \left( \wedge^2 \scrQ_\bfP^\vee \right) = \bP_\bfQ \left( \wedge^2 \scrS_\bfQ \right)$.
Here $V$ is a $5$-dimensional vector space,
$\scrQ_\bfP^\vee$ is the dual of the universal quotient bundle on $\bfP$,
and
$\scrS_{\bfQ}$ is the tautological rank 3 bundle on $\bfQ$.
We have
\begin{align*}
(\varpi_-)_* \left( \cO_\bfF(H) \right) \cong \wedge^2 \scrQ_\bfP
\end{align*}
and
\begin{align*}
(\varpi_+)_* \left( \cO_\bfF(h) \right) \cong \wedge^2 \scrS_\bfQ^\vee,
\end{align*}
whose determinants are given by $\cO_\bfP(2h)$ and $\cO_\bfQ(2H)$ respectively.
Since $\omega_{\bfP} \cong \cO_\bfP(-5h)$, $\omega_{\bfQ} \cong \cO_\bfQ(-5H)$, and $\omega_{\bfF} \cong \cO_\bfF(-3h-3H)$,
we have $\omega_{\bfV_-} \cong \cO_{\bfV_-}$, $\omega_{\bfV_+} \cong \cO_{\bfV_+}$ and $\omega_{\bfV} \cong \cO_\bfV(-2h-2H)$.

First, we adapt several lemmas in \cite{1711.10231} to our situation.
To distinguish vector bundles which are obtained as a pull-back to $\bfF$ from $\bfP$ or $\bfQ$,
we put tilde on the pull-back from $\bfQ$.
By abuse of notation,
we use the same symbol for an object of $D^b(\bfF)$
and
its image in $D^b(\bfV)$ by the push-forward $\iota_*$.

\begin{lem} \label{lem:L1}
$\bhom_{\cO_\bfV} \lb \widetilde{\scrQ}_\bfF, \cO_{\bfF} \lb h + aH \rb \rb \simeq 0$
for integers $-4 \leq a \leq -2$.
\end{lem}
\begin{proof}
We have
\begin{align*}
\bhom_{\cO_\bfV} \lb \widetilde{\scrQ}_\bfF, \cO_{\bfF} \lb h + aH \rb \rb
\simeq
\bh \lb \widetilde{\scrQ}^\vee_\bfF (h+aH) \rb 
\simeq 0,
\end{align*}
where the first and the second isomorphisms follow from \pref{lem:homs}, Borel-Bott-Weil theorem and \cite[Lemma 5.1]{1711.10231} respectively.
\end{proof}

Similarly, one can deduce \pref{lem:L2} and \pref{lem:L3} below from \cite[Lemma 5.2, Lemma 5.3]{1711.10231} by checking that
$\cO_\bfF \lb (a-1)H \rb$,
$\cE^\vee_\bfF \otimes \cE^\prime_\bfF \lb (a-1)h-2H \rb$,
and
$\widetilde{\cF}^\vee_\bfF \otimes \widetilde{\cF}^\prime_\bfF \lb -2h+(a-1)H \rb$
are acyclic as an object of $D^b (\bfF)$.

\begin{lem} \label{lem:L2}
$\bhom_{\cO_\bfV} \lb \cO_\bfF, \cO_{\bfF} \lb h + aH \rb \rb \simeq 0$ for integers $-3 \leq a \leq -1$.
\end{lem}

\begin{lem} \label{lem:L3}
Let $\cE_\bfF, \cE^\prime_\bfF$ be the pull-back to $\bfF$ of vector bundles $\cE, \cE^\prime$ on $\bfP$,
and let $\widetilde{\cF_\bfF}, \widetilde{\cF^\prime_\bfF}$ be the pull-back to $\bfF$ of vector bundles $\cF, \cF^\prime$ on $\bfQ$.
Then we have $\bhom_{\cO_\bfV} \lb \cE_\bfF, \cE^\prime_\bfF \lb ah - H \rb \rb \simeq 0$
and $\bhom_{\cO_\bfV} \lb \widetilde{\cF_\bfF}, \widetilde{\cF^\prime_\bfF} \lb -h + aH \rb \rb \simeq 0$ for all integers $a$.
\end{lem}

The parallel result to the following lemma was tacitly used in \cite{1711.10231}.

\begin{lem} \label{lem:L4}
As an object of $D^b (\bfV)$,
$\cO_\bfF, \widetilde{\scrQ}_\bfF, \scrS_\bfF$, and $\scrS^\vee_\bfF$
are left orthogonal to
$\widetilde{\scrS}^\vee_\bfF \lb h - 2H \rb, \\
\widetilde{\scrS^\vee_\bfF} \lb h-2H \rb, \cO_\bfF \lb 2h-2H \rb$, and $\scrQ_\bfF$
respectively.
\end{lem}

\pref{lem:L5} below
and
the tautological sequence show that
$R_{\cO_\bfF} \widetilde{\scrQ}^\vee_\bfF \simeq \widetilde{\scrS}^\vee_\bfF$
and
$R_{\cO_{\bfF}} \scrS_\bfF \simeq \scrQ_\bfF$
in $D^b (\bfV)$.
 
\begin{lem} \label{lem:L5}
$\bhom_{\cO_\bfV} \lb \widetilde{\scrQ}^\vee_\bfF, \cO_\bfF \rb \simeq V$
and
$\bhom_{\cO_\bfV} \lb \scrS_\bfF, \cO_\bfF \rb \simeq V$.
\end{lem}

Again,
both \pref{lem:L4} and \pref{lem:L5}
follow from \pref{lem:homs} and Borel-Bott-Weil theorem.
\pref{lem:L6} below
and
the exact sequences
\begin{align*}
0 \to \cO_{\bfF}(h-H) \to \scrQ_{\bfF} \to  \widetilde{\scrQ}_{\bfF} \to 0
\end{align*}
and
\begin{align*}
0 \to \scrS_{\bfF} \to \widetilde{\scrS}_{\bfF} \to  \cO_{\bfF}(h-H) \to 0
\end{align*}
obtained in \cite{1711.10231}
show that
$R_{\cO_{\bfF}(h-H)} \widetilde{\scrQ_{\bfF}} \simeq \scrQ_{\bfF}[1]$
and
$L_{\cO_{\bfF}(-h+H)} \widetilde{\scrS^\vee_\bfF} \simeq \scrS^\vee_\bfF$
in $D^b (\bfV)$.
\begin{lem} \label{lem:L6}
$\bhom_{\cO_\bfV} \lb \widetilde{\scrQ}_\bfF, \cO_\bfF(h -H) \rb \simeq \bfk[-1]$
and
$\bhom_{\cO_\bfV} \lb \cO_\bfF(-h+H), \widetilde{\scrS}^\vee_\bfF \rb \simeq \bfk$.
\end{lem}
\begin{proof}
We have
\begin{align*}
\bhom_{\cO_\bfV} \lb \widetilde{\scrQ}_\bfF, \cO_\bfF(h-H) \rb
\simeq
\bh \lb \widetilde{\scrQ}^\vee_\bfF(h-H) \rb
\simeq \bfk[-1],
\end{align*}
where the isomorphisms follow from \pref{lem:homs} and Borel-Bott-Weil theorem.
Similarly, we have
\begin{align*}
\bhom_{\cO_\bfV} \lb \cO_\bfF(-h+H), \widetilde{\scrS}^\vee_\bfF \rb
\simeq
\bh \lb \widetilde{\scrS}^\vee_\bfF(h-H) \rb 
\simeq \bfk.
\end{align*}
\end{proof}

Recall from \cite{Kuz} (cf. also \cite{Kap}) 
\begin{align*} 
D^b(\bfP) = \langle \scrS_{\bfP}(-2h), \cO_{\bfP}(-2h), \scrS_{\bfP}(-h), \cO_{\bfP}(-h), \cdots, \scrS_{\bfP} (2h), \cO_{\bfP} (2h) \rangle,
\end{align*}
and
\begin{align} \label{eq:kEC2}
D^b(\bfQ) = \langle \cO_{\bfQ}, \scrQ_{\bfQ}, \cO_{\bfQ} (H), \scrQ_{\bfQ} (H), \cdots, \cO_{\bfQ} (4H), \scrQ_{\bfQ} (4H) \rangle.
\end{align}
Since $\varphi_\pm$ are blow-ups along the zero-sections,
it follows from \cite{Orl} that 
\begin{align} \label{eq:kSOD1}
D^b(\bfV) = \langle \iota_* \varpi_-^* D^b(\bfP), \iota_* \varpi_-^* D^b(\bfP) (h + H), \Phi_-(D^b(\bfV_-)) \rangle
\end{align}
and
\begin{align} \label{eq:kSOD2}
D^b(\bfV) = \langle \iota_* \varpi_+^* D^b(\bfQ), \iota_* \varpi_+^* D^b(\bfQ) (h + H), \Phi_+(D^b(\bfV_+)) \rangle,
\end{align}
where
\begin{align*}
\Phi_- \coloneqq \left( (-) \otimes \cO_\bfV(2H) \right) \circ \varphi_-^* \colon D^b(\bfV_-) \to D^b(\bfV)
\end{align*}
and
\begin{align*}
\Phi_+ \coloneqq \left( (-) \otimes \cO_\bfV(2h) \right) \circ \varphi_+^* \colon D^b(\bfV_+) \to D^b(\bfV).
\end{align*}
We write $\cO_{i,j} \coloneqq \cO_\bfF(i h + j H)$.
\prettyref{eq:kEC2} and \prettyref{eq:kSOD2} give a semiorthogonal decomposition of the form
\begin{align*}
D^b(\bfV) = \langle &\cO_{0,0}, \widetilde{\scrQ}_{0,0}, \cO_{0,1}, \widetilde{\scrQ}_{0,1}, \cO_{0,2}, \widetilde{\scrQ}_{0,2}, \cO_{0,3}, \widetilde{\scrQ}_{0,3}, \cO_{0,4}, \widetilde{\scrQ}_{0,4} \\
&\cO_{1,1}, \widetilde{\scrQ}_{1,1}, \cO_{1,2}, \widetilde{\scrQ}_{1,2}, \cO_{1,3}, \widetilde{\scrQ}_{1,3}, \cO_{1,4}, \widetilde{\scrQ}_{1,4}, \cO_{1,5}, \widetilde{\scrQ}_{1,5}, \Phi_+(D^b(\bfV_+)) \rangle.
\end{align*}
Since $\omega_\bfV \cong \cO_\bfV(-2h-2H)$,
by mutating the first five terms to the far right, and then $\Phi_+(D^b(\bfV_+))$ five steps to the right, we obtain
\begin{align*}
D^b(\bfV) = \langle &\widetilde{\scrQ}_{0,2}, \cO_{0,3}, \widetilde{\scrQ}_{0,3}, \cO_{0,4}, \widetilde{\scrQ}_{0,4}, \cO_{1,1}, \widetilde{\scrQ}_{1,1}, \cO_{1,2}, \widetilde{\scrQ}_{1,2}, \cO_{1,3} \\
&\widetilde{\scrQ}_{1,3}, \cO_{1,4}, \widetilde{\scrQ}_{1,4}, \cO_{1,5}, \widetilde{\scrQ}_{1,5}, \cO_{2,2}, \widetilde{\scrQ}_{2,2}, \cO_{2,3}, \widetilde{\scrQ}_{2,3}, \cO_{2,4}, \Phi_1(D^b(\bfV_+)) \rangle,
\end{align*}
where
\begin{align*}
\Phi_1 \coloneqq R_{\langle \cO_{2,2}, \widetilde{\scrQ}_{2,2}, \cO_{2,3}, \widetilde{\scrQ}_{2,3}, \cO_{2,4} \rangle} \circ \Phi_+.
\end{align*}
One can easily see that
$\cO_{1,1}$ is orthogonal to $\cO_{0,3}$,
$\widetilde{\scrQ}_{0,3}$,
$\cO_{0,4}$,
and
$\widetilde{\scrQ}_{0,4}$
by \pref{lem:L1} and \pref{lem:L2},
so that
\begin{align*}
D^b(\bfV) = \langle &\widetilde{\scrQ}_{0,2}, \cO_{1,1}, \cO_{0,3}, \widetilde{\scrQ}_{0,3}, \cO_{0,4}, \widetilde{\scrQ}_{0,4}, \widetilde{\scrQ}_{1,1}, \cO_{1,2}, \widetilde{\scrQ}_{1,2}, \cO_{1,3} \\
&\widetilde{\scrQ}_{1,3}, \cO_{2,2}, \cO_{1,4}, \widetilde{\scrQ}_{1,4}, \cO_{1,5}, \widetilde{\scrQ}_{1,5}, \widetilde{\scrQ}_{2,2}, \cO_{2,3}, \widetilde{\scrQ}_{2,3}, \cO_{2,4}, \Phi_1(D^b(\bfV_+)) \rangle .
\end{align*}
By mutating $\widetilde{\scrQ}_{0,2}, \widetilde{\scrQ}_{1,3}, \widetilde{\scrQ}_{1,1}$, and $\widetilde{\scrQ}_{2,2}$ one step to the right,
we obtain by $\widetilde{\scrQ}_{1,1} \cong \widetilde{\scrQ}^\vee_{1,2}$, \pref{lem:L5}, and \pref{lem:L6}
\begin{align*}
D^b(\bfV) = \langle &\cO_{1,1}, \scrQ_{0,2}, \cO_{0,3}, \widetilde{\scrQ}_{0,3}, \cO_{0,4}, \widetilde{\scrQ}_{0,4}, \cO_{1,2}, \widetilde{\scrS}^\vee_{1,2}, \widetilde{\scrQ}_{1,2}, \cO_{1,3} \\
&\cO_{2,2}, \scrQ_{1,3}, \cO_{1,4}, \widetilde{\scrQ}_{1,4}, \cO_{1,5}, \widetilde{\scrQ}_{1,5}, \cO_{2,3}, \widetilde{\scrS}^\vee_{2,3}, \widetilde{\scrQ}_{2,3}, \cO_{2,4}, \Phi_1(D^b(\bfV_+)) \rangle .
\end{align*}
By mutating $\cO_{1,2}$ and $\cO_{2,3}$ four steps to the left, we obtain by \pref{lem:L1}, \pref{lem:L2}, and \pref{lem:L6}
\begin{align*}
D^b(\bfV) = \langle &\cO_{1,1}, \scrQ_{0,2}, \cO_{1,2}, \cO_{0,3}, \scrQ_{0,3}, \cO_{0,4}, \widetilde{\scrQ}_{0,4}, \widetilde{\scrS}^\vee_{1,2}, \widetilde{\scrQ}_{1,2}, \cO_{1,3} \\
&\cO_{2,2}, \scrQ_{1,3}, \cO_{2,3}, \cO_{1,4}, \scrQ_{1,4}, \cO_{1,5}, \widetilde{\scrQ}_{1,5}, \widetilde{\scrS}^\vee_{2,3}, \widetilde{\scrQ}_{2,3}, \cO_{2,4}, \Phi_1(D^b(\bfV_+)) \rangle .
\end{align*}
One can easily see that $\widetilde{\scrS}^\vee_{1,2}$ is orthogonal to $\cO_{0,4}$ and $\widetilde{\scrQ}_{0,4}$ by \pref{lem:L4},
so that
\begin{align*}
D^b(\bfV) = \langle &\cO_{1,1}, \scrQ_{0,2}, \cO_{1,2}, \cO_{0,3}, \scrQ_{0,3}, \widetilde{\scrS}^\vee_{1,2}, \cO_{0,4}, \widetilde{\scrQ}_{0,4}, \widetilde{\scrQ}_{1,2}, \cO_{1,3} \\
&\cO_{2,2}, \scrQ_{1,3}, \cO_{2,3}, \cO_{1,4}, \scrQ_{1,4}, \widetilde{\scrS}^\vee_{2,3}, \cO_{1,5}, \widetilde{\scrQ}_{1,5}, \widetilde{\scrQ}_{2,3}, \cO_{2,4}, \Phi_1(D^b(\bfV_+)) \rangle .
\end{align*}
By mutating $\cO_{0,3}$ and $\cO_{1,4}$ two steps to the right,
$\cO_{1,3}$ and $\cO_{2,4}$ three steps to the left,
and then
$\cO_{0,4}$ and $\cO_{1,5}$ two steps to the right,
we obtain by \pref{lem:L5} and \pref{lem:L6}
\begin{align*}
D^b(\bfV) = \langle &\cO_{1,1}, \scrQ_{0,2}, \cO_{1,2}, \scrS_{0,3}, \scrS^\vee_{1,2}, \cO_{0,3}, \cO_{1,3}, \scrS_{0,4}, \scrS^\vee_{1,3}, \cO_{0,4} \\
&\cO_{2,2}, \scrQ_{1,3}, \cO_{2,3}, \scrS_{1,4}, \scrS^\vee_{2,3}, \cO_{1,4}, \cO_{2,4}, \scrS_{1,5}, \scrS^\vee_{2,4}, \cO_{1,5}, \Phi_1(D^b(\bfV_+)) \rangle .
\end{align*}
By mutating $\cO_{1,1}$ to the far right, and then $\Phi_1(D^b(\bfV_+))$ one step to the right, we obtain
\begin{align*}
D^b(\bfV) = \langle &\scrQ_{0,2}, \cO_{1,2}, \scrS_{0,3}, \scrS^\vee_{1,2}, \cO_{0,3}, \cO_{1,3}, \scrS_{0,4}, \scrS^\vee_{1,3}, \cO_{0,4}, \cO_{2,2} \\
&\scrQ_{1,3}, \cO_{2,3}, \scrS_{1,4}, \scrS^\vee_{2,3}, \cO_{1,4}, \cO_{2,4}, \scrS_{1,5}, \scrS^\vee_{2,4}, \cO_{1,5}, \cO_{3,3}, \Phi_2(D^b(\bfV_+)) \rangle ,
\end{align*}
where
\begin{align*}
\Phi_2 \coloneqq R_{\langle \cO_{3,3} \rangle} \circ \Phi_1.
\end{align*}
By \pref{lem:L2},
\pref{lem:L3},
and
\pref{lem:L4}, we obtain
\begin{align*}
D^b(\bfV) = \langle &\scrQ_{0,2}, \cO_{1,2}, \scrS^\vee_{1,2}, \cO_{2,2}, \scrS_{0,3}, \cO_{0,3}, \cO_{1,3}, \scrS^\vee_{1,3}, \scrQ_{1,3}, \cO_{2,3} \\
&\scrS^\vee_{2,3}, \cO_{3,3}, \scrS_{0,4}, \cO_{0,4}, \scrS_{1,4}, \cO_{1,4}, \cO_{2,4}, \scrS^\vee_{2,4}, \scrS_{1,5}, \cO_{1,5}, \Phi_2(D^b(\bfV_+)) \rangle.
\end{align*}
By mutating $\Phi_2(D^b(\bfV_+))$ ten steps to the left, and then last ten terms to the far left, we obtain
\begin{align*}
D^b(\bfV) = \langle &\scrS^\vee_{0,1}, \cO_{1,1}, \scrS_{-2,2}, \cO_{-2,2}, \scrS_{-1,2}, \cO_{-1,2}, \cO_{0,2}, \scrS^\vee_{0,2}, \scrS_{-1,3}, \cO_{-1,3} \\
&\scrQ_{0,2}, \cO_{1,2}, \scrS^\vee_{1,2}, \cO_{2,2}, \scrS_{0,3}, \cO_{0,3}, \cO_{1,3}, \scrS^\vee_{1,3}, \scrQ_{1,3}, \cO_{2,3}, \Phi_3(D^b(\bfV_+)) \rangle ,
\end{align*}
where
\begin{align*}
\Phi_3 \coloneqq L_{\langle \scrS^\vee_{2,3}, \cO_{3,3}, \scrS_{0,4}, \cO_{0,4}, \scrS_{1,4}, \cO_{1,4}, \cO_{2,4}, \scrS^\vee_{2,4}, \scrS_{1,5}, \cO_{1,5} \rangle} \circ \Phi_2.
\end{align*}
By \pref{lem:L3}, we obtain
\begin{align*}
D^b(\bfV) = \langle &\scrS^\vee_{0,1}, \cO_{1,1}, \scrS_{-2,2}, \cO_{-2,2}, \scrS_{-1,2}, \cO_{-1,2}, \cO_{0,2}, \scrS^\vee_{0,2}, \scrQ_{0,2}, \cO_{1,2} \\
&\scrS^\vee_{1,2}, \cO_{2,2}, \scrS_{-1,3}, \cO_{-1,3}, \scrS_{0,3}, \cO_{0,3}, \cO_{1,3}, \scrS^\vee_{1,3}, \scrQ_{1,3}, \cO_{2,3}, \Phi_3(D^b(\bfV_+)) \rangle .
\end{align*}
By mutating $\scrQ_{0,2}$ and $\scrQ_{1,3}$ two steps to the left, 
the first two terms to the far right,
and then $\Phi_3(D^b(\bfV_+))$ two steps to the right, 
we obtain by $\scrS^\vee_{0,0} \simeq \scrS_{1,0}$, \pref{lem:L4}, and \pref{lem:L6}
\begin{align} \label{eq:kSOD3}
\begin{split}
D^b(\bfV) = \langle &\scrS_{-2,2}, \cO_{-2,2}, \scrS_{-1,2}, \cO_{-1,2}, \scrS_{0,2}, \cO_{0,2}, \scrS_{1,2}, \cO_{1,2}, \scrS_{2,2}, \cO_{2,2} \\
&\scrS_{-1,3}, \cO_{-1,3}, \scrS_{0,3}, \cO_{0,3}, \scrS_{1,3}, \cO_{1,3}, \scrS_{2,3}, \cO_{2,3}, \scrS_{3,3}, \cO_{3,3}, \Phi_4(D^b(\bfV_+)) \rangle ,
\end{split}
\end{align}
where
\begin{align*}
\Phi_4 \coloneqq R_{\langle \scrS^\vee_{2,3}, \cO_{3,3} \rangle} \circ \Phi_3.
\end{align*}
By comparing \eqref{eq:kSOD3} with \eqref{eq:kSOD1},
we obtain a derived equivalence
\begin{align*}
\Phi \coloneqq \Phi_-^! \circ \Phi_4 \colon D^b(\bfV_+) \xrightarrow{\sim} D^b(\bfV_-),
\end{align*}
where
\begin{align*}
\Phi_-^!(-) \coloneqq (\varphi_-)_* \circ \left( (-) \otimes \cO_\bfV(-2H) \right)
\colon D^b(\bfV) \to D^b(\bfV_-)
\end{align*}
is the left adjoint functor of $\Phi_-$.

\section{Mukai flop}
For $n \ge 2$,
let $P$ and $Q$ be the maximal parabolic subgroups
of the simple Lie group of type $A_n$
associated with the crossed Dynkin diagrams
\dynkin{A}{x*.*}
and
\dynkin{A}{*.*x}.
The corresponding homogeneous spaces are the projective spaces
$\bfP = \bP V,$
$\bfQ = \bP V^\vee,$
and the partial flag variety $\bfF = \operatorname{F} \left( 1, n; V \right)$,
where $V$ is an $(n+1)$-dimensional vector space.
Since 
$
\omega_{\bfP} \cong \cO(-(n+1)h),
$
$
\omega_{\bfQ} \cong \cO(-(n+1)H),
$
and
$
\omega_{\bfF} \cong \cO(-nh-nH),
$
we have
$
\omega_{\bfV_-} \cong \cO_{\bfV_-},
$
$
\omega_{\bfV_+} \cong \cO_{\bfV_+},
$
and
$
\omega_{\bfV} \cong \cO(-(n-1)h-(n-1)H).
$

\begin{lem} \label{lm:mnExt1}
$\cO_{\bfF}(- i h + jH)$ and
$\cO_{\bfF}(- (i +1) h + (j - 1)H)$
are acyclic for $1 \leq j \leq n-1$ and $1 \leq i \leq n-j$.
\end{lem}

\begin{proof}
Since
$j-n \leq -i \leq -1$
and
$j-n-1 \leq -i-1 \leq -2$,
the derived push-foward of
$\cO_{\bfF}(- i h + jH)$
and
$\cO_{\bfF}(- (i + 1) h + (j - 1)H)$
vanish by \cite[Exercise III.8.4]{Har}
unless
$i = n - 1$ and $j = 1$,
in which case the acyclicity of
$\cO_{\bfF}(-nh)$
is obvious.
\end{proof}

\begin{lem} \label{lm:mnExt2}
$
 \bhom_{\cO_\bfV} \left( \cO_{\bfF}(i h - jH), \cO_{\bfF} \right) \simeq 0
$
for $1 \leq j \leq n-1$ and $1 \leq i \leq n-j$.
\end{lem}
\begin{proof}
We have
\begin{align*}
\bhom_{\cO_\bfV} \left( \cO_{\bfF}(i h -jH), \cO_{\bfF} \right)
&\simeq
\mathbf{h} \left( \left\{ \cO_{\bfF}(-i h + jH) \to \cO_{\bfF}(-(i + 1)h + (j - 1)H) \right\} \right),
\end{align*}
which vanishes
by \prettyref{lm:mnExt1}.
\end{proof}

Recall from \cite{Bei} that 
\begin{align} \label{eq:mnEC1}
D^b(\bfP) = \langle \cO_{\bfP}, \cO_\bfP(h), \cdots, \cO_\bfP(nh) \rangle
\end{align}
and
\begin{align} \label{eq:mnEC2}
D^b(\bfQ) = \langle \cO_{\bfQ}, \cO_\bfQ(H), \cdots, \cO_\bfQ(nH) \rangle.
\end{align}

Since $\varphi_\pm$ are blow-ups
along the zero-sections,
it follows from \cite{Orl} that 
\begin{align} \label{eq:mnSOD1}
D^b(\bfV) = \langle \iota_* \varpi_-^* D^b(\bfP), \cdots, \iota_* \varpi_-^* D^b(\bfP) \otimes \cO_{\bfV}((n-2)H), \Phi_-(D^b(\bfV_-)) \rangle
\end{align}
and
\begin{align} \label{eq:mnSOD2}
D^b(\bfV) = \langle \iota_* \varpi_+^* D^b(\bfQ), \cdots, \iota_* \varpi_+^* D^b(\bfQ) \otimes \cO_{\bfV}((n-2)h), \Phi_+(D^b(\bfV_+)) \rangle,
\end{align}
where
\begin{align*}
\Phi_- \coloneqq \left( (-) \otimes \cO_\bfV((n-1)H) \right) \circ \varphi_-^* \colon D^b(\bfV_-) \to D^b(\bfV)
\end{align*}
and
\begin{align*}
\Phi_+ \coloneqq \left( (-) \otimes \cO_\bfV((n-1)h) \right) \circ \varphi_+^* \colon D^b(\bfV_+) \to D^b(\bfV).
\end{align*}
We write $\cO_{i,j} \coloneqq \cO_\bfF(i h + j H)$.
\pref{eq:mnEC1} and \pref{eq:mnSOD1} give
a semiorthogonal decomposition of the form
\begin{align*}
 D^b(\bfV) = \langle \cA_0, \Phi_-(D^b(\bfV_-)) \rangle
\end{align*}
where $\cA_0$ is given by
\begin{align} \label{eq:mnEC0}
\begin{matrix}
 \cO_{0,0} & \cO_{1,0} & \cdots & \cO_{n-2,0} & \cO_{n-1,0} & \cO_{n,0} \\
 & \cO_{1,1} & \cdots & \cO_{n-2,1} & \cO_{n-1,1} & \cO_{n,1} & \cO_{n+1,1} \\
 && \ddots & \vdots & \vdots & \vdots & \vdots & \ddots \\
 &&& \cO_{n-2,n-2} & \cO_{n-1,n-2} & \cO_{n,n-2} & \cO_{n+1,n-2} & \cdots & \cO_{2n-2,n-2}.
\end{matrix}
\end{align}
Note from \prettyref{lm:mnExt2} that
there are no morphisms from right to left in \pref{eq:mnEC0}. 
Since $\omega_\bfV \cong \cO_{-(n-1),-(n-1)}$,
by mutating first
\begin{align*}
\begin{matrix}
 \cO_{0,0} & \cO_{1,0} & \cdots & \cO_{n-2,0} \\
 & \cO_{1,1} & \cdots & \cO_{n-2,1} \\
 && \ddots & \vdots \\
 &&& \cO_{n-2,n-2}
\end{matrix}
\end{align*}
to the far right, and then
$\Phi_-(D^b(\bfV_-))$
to the far right, we obtain
\begin{align*}
 D^b(\bfV) = \langle \cA_1, \Phi_1(D^b(\bfV_-)) \rangle,
\end{align*}
where
\begin{align*}
\Phi_1(D^b(\bfV_-)) \coloneqq 
R_{\langle \cO_{n-1,n-1}, \cdots, \cO_{2n-3,2n-3} \rangle} \circ \Phi_-
\end{align*}
and $\cA_1$ is given by
\begin{align*}
\begin{matrix}
 \cO_{n-1,0} & \cO_{n,0} \\
 \cO_{n-1,1} & \cO_{n,1} & \cO_{n+1,1} \\
 \vdots & \vdots & \vdots & \ddots \\
 \cO_{n-1,n-2} & \cO_{n,n-2} & \cO_{n+1,n-2} & \cdots & \cO_{2n-3,n-2} & \cO_{2n-2,n-2} \\
 \cO_{n-1,n-1} & \cO_{n,n-1} & \cO_{n+1,n-1} & \cdots & \cO_{2n-3,n-1} \\
 & \cO_{n,n} & \cO_{n+1,n} & \cdots & \cO_{2n-3,n} \\
 && \cO_{n+1,n+1} & \cdots & \cO_{2n-3,n+1} \\
 &&& \ddots & \vdots \\
 &&&& \cO_{2n-3,2n-3}.
\end{matrix}
\end{align*}
By mutating $\Phi_1(D^b(\bfV_-))$ one step to the left,
and then $\cO_{2n-2,n-2}$ to the far left, we obtain
\begin{align} \label{eq:mnSOD3}
 D^b(\bfV) = \langle \cA_2, \Phi_2(D^b(\bfV_-)) \rangle,
\end{align}
where
\begin{align*}
 \Phi_2(D^b(\bfV_-)) \coloneqq L_{\cO_{2n-2,n-2}} \circ \Phi_1
\end{align*}
and $\cA_2$ is given by
\begin{align*}
\begin{matrix}
 \cO_{n-1,-1} \\
 \cO_{n-1,0} & \cO_{n,0} \\
 \cO_{n-1,1} & \cO_{n,1} & \cO_{n+1,1} \\
 \vdots & \vdots & \vdots & \ddots \\
 \cO_{n-1,n-2} & \cO_{n,n-2} & \cO_{n+1,n-2} & \cdots & \cO_{2n-3,n-2} \\
 \cO_{n-1,n-1} & \cO_{n,n-1} & \cO_{n+1,n-1} & \cdots & \cO_{2n-3,n-1} \\
 & \cO_{n,n} & \cO_{n+1,n} & \cdots & \cO_{2n-3,n} \\
 && \cO_{n+1,n+1} & \cdots & \cO_{2n-3,n+1} \\
 &&& \ddots & \vdots \\
 &&&& \cO_{2n-3,2n-3}.
\end{matrix}
\end{align*}
By comparing \pref{eq:mnSOD3}
with \pref{eq:mnEC2} and \pref{eq:mnSOD2},
we obtain a derived equivalence
\begin{align*} 
\Phi \coloneqq (\varphi_+)_* \circ ((-) \otimes \cO_{-(2n-2),0}) \circ \Phi_2
 \colon D^b(\bfV_-) \xrightarrow{\sim} D^b(\bfV_+).
\end{align*}

\section{Standard flop}
For $n \ge 1$,
let $P$ and $Q$ be the maximal parabolic subgroups of the semisimple Lie group 
$G = \operatorname{SL}(V) \times  \operatorname{SL}(V^\vee)$
associated with the crossed Dynkin diagram
$\dynkin{A}{x*.*}
\oplus
\dynkin{A}{*.**}$
and
$\dynkin{A}{*.**}
\oplus
\dynkin{A}{*.*x.}.$
The corresponding homogeneous spaces are the projective spaces
$\bfP = \bP V$,
$\bfQ = \bP V^\vee,$
and their product $\bfF = \bP V \times \bP V^\vee$.
Since 
$\omega_{\bfP} \cong \cO(-(n+1)h)$,
$\omega_{\bfQ} \cong \cO(-(n+1)H)$,
and
$\omega_{\bfF} \cong \cO(-(n+1)h-(n+1)H)$,
we have
$\omega_{\bfV_-} \cong \cO_{\bfV_-}$,
$\omega_{\bfV_+} \cong \cO_{\bfV_+}$,
and
$\omega_{\bfV} \cong \cO(-nh-nH)$.

\begin{lem} \label{lm:stnExt2}
$\bhom_{\cO_\bfV} \left( \cO_{\bfF}(i h - jH), \cO_{\bfF} \right) \simeq 0$
for $1 \leq j \leq n-1$ and $1 \leq i \leq n-j$.
\end{lem}

\begin{proof}
We have
\begin{align*}
\bhom_{\cO_\bfV} \left( \cO_{\bfF}(i h -jH), \cO_{\bfF} \right)
&\simeq
\mathbf{h} \left( \left\{ \cO_{\bfF}(-i h + jH) \to \cO_{\bfF}(-(i + 1)h + (j - 1)H) \right\} \right),
\end{align*}
which vanishes
for $1 \leq i \leq n-j \le n - 1$.
\end{proof}

It follows from \cite{Orl} that 
\begin{align} \label{eq:stnSOD1}
D^b(\bfV) = \langle \iota_* \varpi_-^* D^b(\bfP), \cdots, \iota_* \varpi_-^*D^b(\bfP) \otimes \cO((n-1)(h+H)), \Phi_-(D^b(\bfV_-)) \rangle
\end{align}
and
\begin{align} \label{eq:stnSOD2}
D^b(\bfV) = \langle \iota_* \varpi_+^* D^b(\bfQ), \cdots, \iota_* \varpi_+^* D^b(\bfQ) \otimes \cO((n-1)(h+H)), \Phi_+(D^b(\bfV_+)) \rangle,
\end{align}
where
\begin{align*}
\Phi_- \coloneqq (-) \otimes \cO_\bfV(n(h+H)) \circ \varphi_-^* \colon D^b(\bfV_-) \to D^b(\bfV)
\end{align*}
and
\begin{align*}
\Phi_+ \coloneqq (-) \otimes \cO_\bfV(n(h+H)) \circ \varphi_+^* \colon D^b(\bfV_+) \to D^b(\bfV).
\end{align*}
We write $\cO_{i,j} \coloneqq \cO_\bfF(i h + j H)$.
\prettyref{eq:mnEC1} and \prettyref{eq:stnSOD1} give
a semiorthogonal decomposition of the form
\begin{align*}
 D^b(\bfV) = \langle \cA_0, \Phi_-(D^b(\bfV_-)) \rangle
\end{align*}
where $\cA_0$ is given by
\begin{align} \label{eq:stnEC0}
\begin{matrix}
 \cO_{0,0} & \cO_{1,0} & \cdots & \cO_{n-2,0} & \cO_{n-1,0} & \cO_{n,0} \\
 & \cO_{1,1} & \cdots & \cO_{n-2,1} & \cO_{n-1,1} & \cO_{n,1} & \cO_{n+1,1} \\
 && \ddots & \vdots & \vdots & \vdots & \vdots & \ddots \\
 &&& \cO_{n-2,n-2} & \cO_{n-1,n-2} & \cO_{n,n-2} & \cO_{n+1,n-2} & \cdots & \cO_{2n-2,n-2} \\
 &&&& \cO_{n-1,n-1} & \cO_{n,n-1} & \cO_{n+1,n-1}  & \cdots & \cO_{2n-2,n-1} & \cO_{2n-1,n-1}.
\end{matrix}
\end{align}
Note from \prettyref{lm:stnExt2} that
there are no morphisms from right to left in \pref{eq:stnEC0}. 
Since $\omega_\bfV \cong \cO_\bfV(-nh-nH)$, by mutating first
\begin{align*}
\begin{matrix}
 \cO_{0,0} & \cO_{1,0} & \cdots & \cO_{n-2,0} \\
 & \cO_{1,1} & \cdots & \cO_{n-2,1} \\
 && \ddots & \vdots \\
 &&& \cO_{n-2,n-2}
\end{matrix}
\end{align*}
to the far right, and then
$\Phi_-(D^b(\bfV_-))$
to the far right, we obtain
\begin{align*}
 D^b(\bfV) = \langle \cA_1, \Phi_1(D^b(\bfV_-)) \rangle,
\end{align*}
where
\begin{align*}
\Phi_1(D^b(\bfV_-)) \coloneqq 
R_{\langle \cO_{n,n}, \cdots, \cO_{2n-2,2n-2} \rangle} \circ \Phi_-
\end{align*}
and $\cA_1$ is given by
\begin{align*}
\begin{matrix}
 \cO_{n-1,0} & \cO_{n,0} \\
 \cO_{n-1,1} & \cO_{n,1} & \cO_{n+1,1} \\
 \vdots & \vdots & \vdots & \ddots \\
 \cO_{n-1,n-1} & \cO_{n,n-1} & \cO_{n+1,n-1} & \cdots & \cO_{2n-2,n-1} & \cO_{2n-1,n-1} \\
 & \cO_{n,n} & \cO_{n+1,n} & \cdots & \cO_{2n-2,n} \\
 && \cO_{n+1,n+1} & \cdots & \cO_{2n-2,n+1} \\
 &&& \ddots & \vdots \\
 &&&& \cO_{2n-2,2n-2}.
\end{matrix}
\end{align*}
By mutating $\Phi_1(D^b(\bfV_-))$ one step to the left,
and then $\cO_{2n-1,n-1}$ to the far left, we obtain
\begin{align} \label{eq:stnSOD3}
 D^b(\bfV) = \langle \cA_2, \Phi_2(D^b(\bfV_-)) \rangle,
\end{align}
where
\begin{align*}
 \Phi_2(D^b(\bfV_-)) \coloneqq L_{\cO_{2n-1,n-1}} \circ \Phi_1
\end{align*}
and $\cA_2$ is given by
\begin{align*}
\begin{matrix}
 \cO_{n-1,-1} \\
 \cO_{n-1,0} & \cO_{n,0} \\
 \cO_{n-1,1} & \cO_{n,1} & \cO_{n+1,1} \\
 \vdots & \vdots & \vdots & \ddots \\
 \cO_{n-1,n-1} & \cO_{n,n-1} & \cO_{n+1,n-1} & \cdots & \cO_{2n-2,n-1} \\
 & \cO_{n,n} & \cO_{n+1,n} & \cdots & \cO_{2n-2,n} \\
 && \cO_{n+1,n+1} & \cdots & \cO_{2n-2,n+1} \\
 &&& \ddots & \vdots \\
 &&&& \cO_{2n-2,2n-2}.
\end{matrix}
\end{align*}
By comparing \pref{eq:stnSOD3}
with \pref{eq:mnEC2} and \pref{eq:stnSOD2},
we obtain a derived equivalence
\begin{align*} 
\Phi \coloneqq (\varphi_+)_* \circ ((-) \otimes \cO_{-(2n-1),0}) \circ \Phi_2
 \colon D^b(\bfV_-) \xrightarrow{\sim} D^b(\bfV_+).
\end{align*}

\begin{rmk}
The way of presenting our proof in Section 4 and 5 is called chess game by some authors \cite{JL, Tho}.
\end{rmk}


\end{document}